\documentclass{siamltex}
\usepackage[dvips]{graphicx}
\usepackage{amsfonts}
\usepackage{amssymb}
\usepackage{color}
\usepackage{subfigure}
\usepackage{amsmath,latexsym}

\newtheorem{Assumption}{Assumption}[section]

\newcommand{\lbl}[1]{\label{#1}}

\def\eps{\varepsilon}

\def\cond{\operatorname{cond}}

\def\Th{\mathcal
{T}_h}
\def\T{\mathcal{T}}
\def\R{\Bbb{R}}

\def\bp{\mathbf{p}}

\def\bA{\mathbf{A}}
\def\bB{\mathbf{B}}
\def\bD{\mathbf{D}}
\def\bn{\mathbf{n}}
\def\bv{\mathbf{v}}

\def\bfp{\mathbf{p}}

\title{On surface meshes induced by level set functions}
\author{Maxim A. Olshanskii\thanks{Department of Mathematics, University of Houston, Houston, Texas 77204-3008 and Dept. Mechanics and Mathematics, Moscow State
University, Moscow 119899
(molshan@math.uh.edu).}
\and Arnold Reusken\thanks{Institut f\"ur Geometrie und Praktische  Mathematik, RWTH-Aachen
University, D-52056 Aachen, Germany (reusken@igpm.rwth-aachen.de,xu@igpm.rwth-aachen.de).}
\and Xianmin Xu$^\dag$\thanks{LSEC, Institute of Computational
  Mathematics and Scientific/Engineering Computing,
  NCMIS, AMSS, Chinese Academy of Sciences, Beijing 100190, China (xmxu@lsec.cc.ac.cn).}
}

\begin{document}

\maketitle

\renewcommand{\thefootnote}{\arabic{footnote}}
\begin{abstract}
The zero level set of a  continuous piecewise-affine function with respect to a
consistent tetrahedral subdivision of a domain in  $\mathbb{R}^3$ is
a piecewise-planar hyper-surface. We prove that if a family of consistent tetrahedral subdivions satisfies the \emph{minimum angle condition}, then  after a simple postprocessing
this zero level set   becomes a consistent surface triangulation
which satisfies the \emph{maximum angle condition}.
We treat an application of this result to the numerical solution of PDEs posed on surfaces, using a $P_1$ finite element space on such a surface triangulation. For this finite element space we derive optimal interpolation error bounds.
We  prove that the diagonally scaled mass matrix is well-conditioned, uniformly with respect to $h$. Furthermore, the
 issue of conditioning of the stiffness matrix is addressed.
\end{abstract}



\pagestyle{myheadings}
\thispagestyle{plain}
\markboth{M. A. Olshanskii, A. Reusken, X. Xu  }{On surface meshes induced by level set functions}
\section{Introduction}

Surface triangulations occur in, for example, visualization, shape optimization,
surface restoration  and in  applications where differential equations posed
on surfaces are treated numerically.
Hence,  properties of surface triangulations such as shape regularity and angle conditions are of interest. For example, angle conditions are closely related to approximation properties and stability of corresponding finite
elements \cite{Babuska76,Ciarlet}.

 In this article, we are interested in the properties of a surface triangulation
 if one considers the zero level of a continuous piecewise-affine function with respect to a
consistent tetrahedral subdivision of a domain in $\mathbb{R}^3$.  The zero level of a  piecewise-affine function is
a piecewise-planar hyper-surface
consisting of triangles and quadrilaterals. Each  quadrilateral can be divided into two triangles in such a
way that the resulting surface triangulation satisfies the following property proved in this paper:
if the volume tetrahedral subdivision satisfies a  minimum angle condition, then the corresponding surface triangulation satisfies
a maximum angle condition. We show that the maximum angle occuring in the surface triangulation can be bounded by a constant $\phi_{\max} < \pi$ that depends only on a stability constant for the family of tetrahedral subdivisions.

The paper also discusses a few implications of this property for the numerical solution of surface partial differential equations.
Numerical methods for surface PDEs are studied in e.g., \cite{Dziuk07,Demlow06,Dziuk88,Demlow09,OlshanskiiReusken08,XuZhao03}.  We derive optimal
approximation properties of $P_1$ finite element functions with respect to the surface
triangulation and  a uniform bound for the condition number of the scaled mass matrix. We also show that the condition number of the (scaled) stiffness matrix can be very large and is sensitive to the
distribution of the vertices of tetrahedra
close to the surface.  Some numerical examples illustrate the analysis of the paper.

\section{Surface meshes induced by  regular bulk triangulations} \label{sectSUPG}
Consider a smooth surface $\Gamma$ in three dimensional space. For simplicity, we assume that $\Gamma$ is connected and has no boundary.
Let $\Omega\subset\mathbb{R}^3$ be a bulk domain which contains $\Gamma$.
Let $\{\mathcal{T}_h\}_{h>0}$ be a family of tetrahedral triangulations of
the domain $\Omega$. These triangulations
are assumed to be regular, consistent and stable, cf.~\cite{Ciarlet}. To simplify the presentation, we assume that this family of triangulations is \emph{quasi-uniform}. The latter assumption, however, is not essential for our analysis.

We assume that for each $\mathcal{T}_h$  an approximation  of
$\Gamma$, denoted by $\Gamma_h$, is given which  is a connected $C^{0,1}$ surface without  boundary. In our analysis we assume $\Gamma_h$ to be consistent with $\mathcal{T}_h$ is the sense as explained in the following definition.
\begin{definition} \label{defconsistent} For any
tetrahedron $S_T\in\mathcal{T}_h$ such that $\mathrm{meas}_2(S_T\cap\Gamma_h)>0$ define
$T=S_T\cap\Gamma_h$. If every $T$  is a \textit{planar},
then the surface approximation  $\Gamma_h$ is called \emph{consistent with the outer triangulation} $\mathcal{T}_h$.
\end{definition}

If $\Gamma_h$ is consistent with $\mathcal{T}_h$, then every segment $T=S_T\cap\Gamma_h$ is either a triangle or a quadrilateral. Each quadrilateral segment can be divided into two triangles, so we may assume that every $T$ is a triangle.

Let $\mathcal{F}_h$ be the set of all  triangular segments $T$, then $\Gamma_h$ can
be decomposed as
\begin{equation} \label{defgammah}
 \Gamma_h=\bigcup\limits_{T\in\mathcal{F}_h} T.
\end{equation}
\begin{Assumption}
 In the remainder of this paper we assume that $\Gamma_h$ is a connected $C^{0,1}$ surface without  boundary that is consistent with the outer triangulation $\mathcal{T}_h$.
\end{Assumption}

The most prominent example of such a surface triangulation is obtained in the context of level set techniques. Assume that $\Gamma$ is represented as the zero level of a level set function $\phi$ and that $\phi_h$ is a continuous linear finite element approximation on the outer  tetrahedral triangulation $\mathcal{T}_h$. Then if we define $\Gamma_h$ to be the zero level of $\phi_h$ then $\Gamma_h$ consists of piecewise planar segments and is consistent with $\mathcal{T}_h$. As an example,
 consider a sphere $\Gamma$, represented as the zero level  of its signed distance function. For $\phi_h$ we take  the piecewise linear nodal interpolation of this distance function on a uniform tetrahedral
triangulation $\mathcal{T}_h$ of a domain that contains $\Gamma$. The zero level of this interpolant defines $\Gamma_h$ and is illustrated in Fig.~\ref{fig:tri}.
\begin{figure}[ht!]
\begin{center}
\centering
  \includegraphics[width=0.45\textwidth]{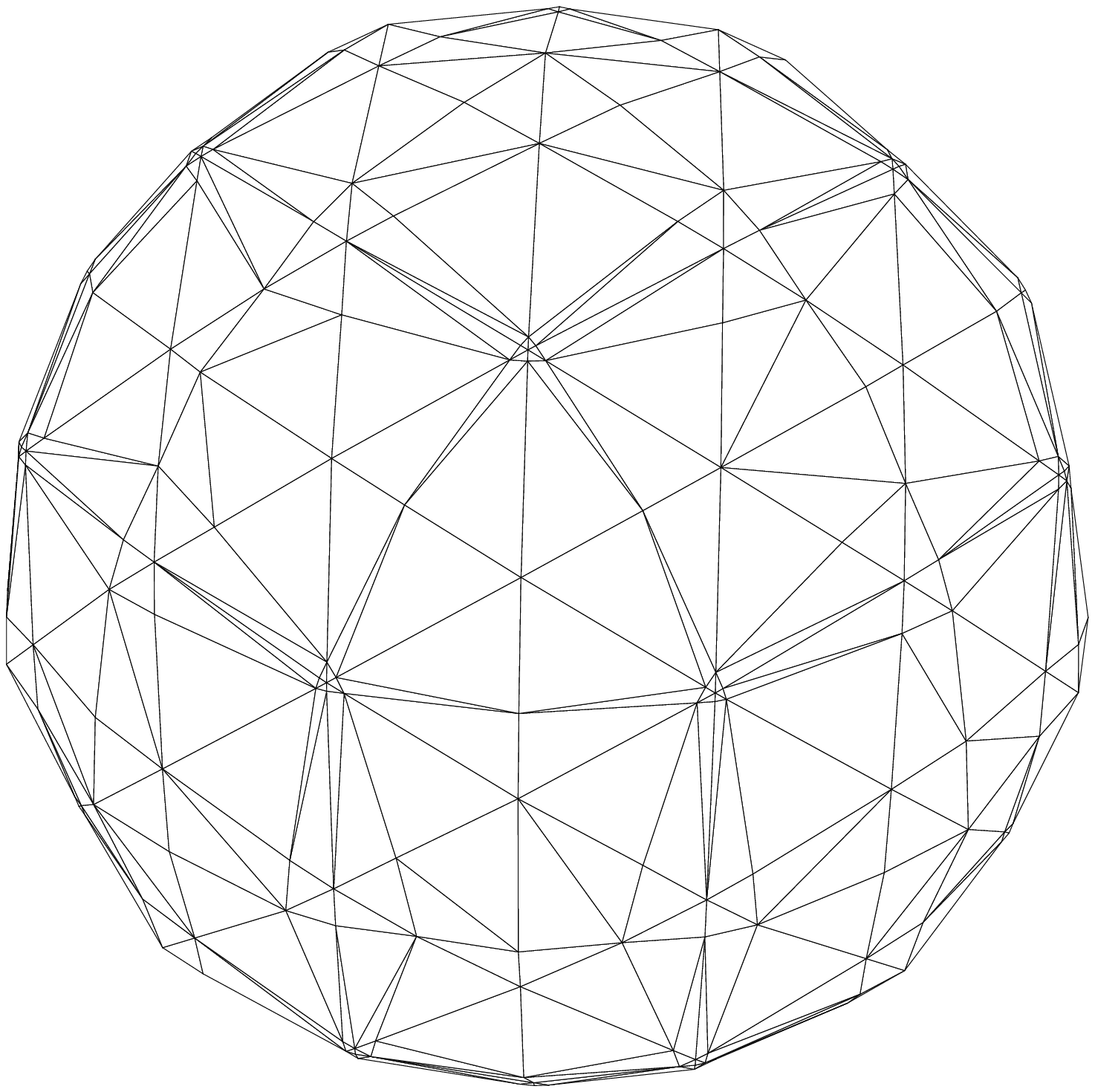}
  \includegraphics[width=0.45\textwidth]{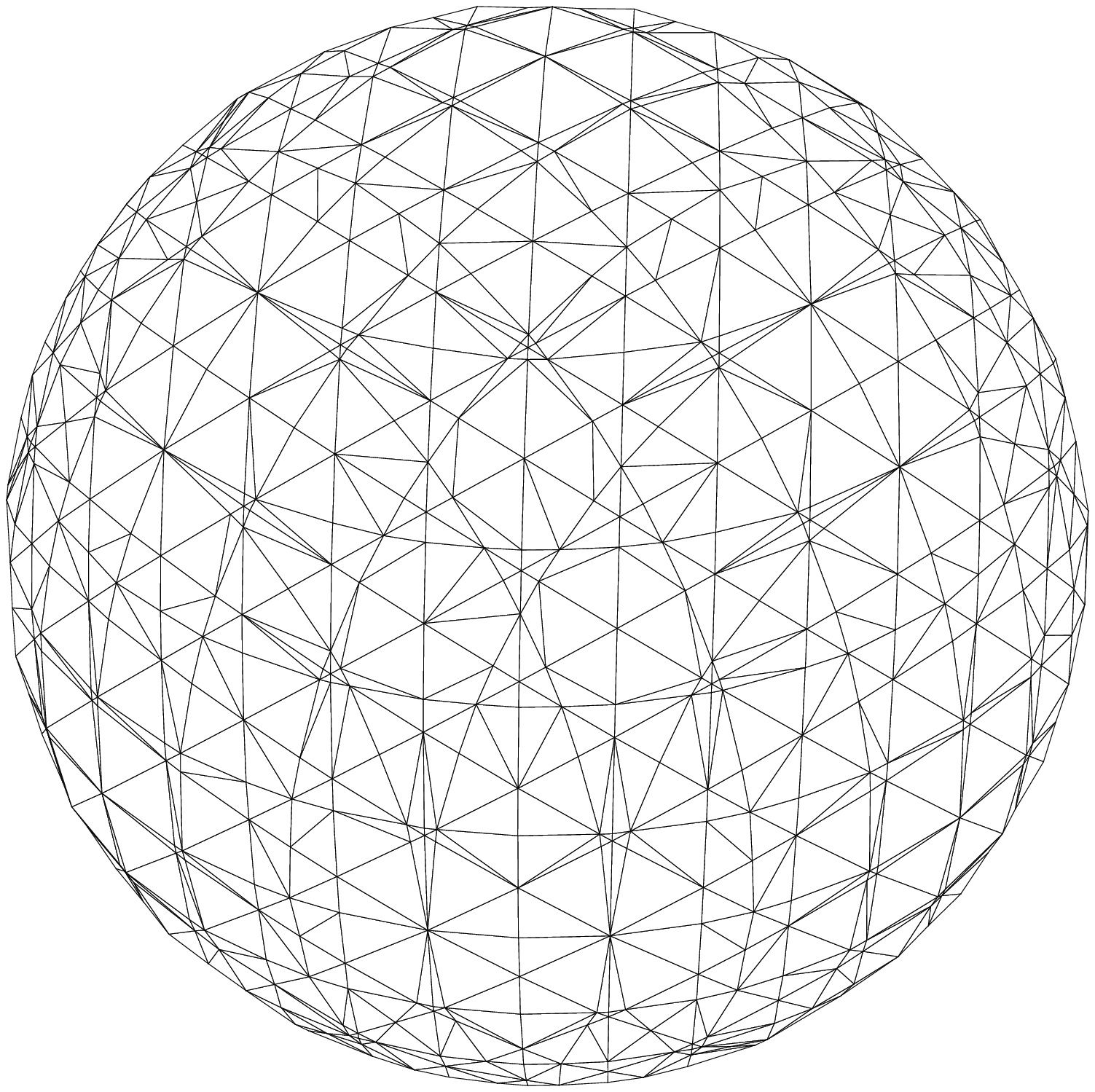}
\caption{Approximate interface $\Gamma_h$ for an example of a sphere, resulting from a coarse tetrahedral triangulation (left) and after one refinement (right).}
\label{fig:tri}
\end{center}
\end{figure}

In the setting of level set methods, such surface triangulations induced by a finite element level set function on a regular outer tetrahedral triangulation are very natural and easy to construct. A surface triangulation $\Gamma_h$ that is consistent with the outer triangulation may be the result of another method than the level set method. In the remainder we only need that  $\Gamma_h$ is consistent with the outer triangulation and not that it is generated by a level set technique.

Note that the triangulation $\mathcal{F}_h$  is \emph{not} necessarily  regular, i.e. elements $T\in\mathcal{F}_h $ may have
\emph{very small inner angles} and the \emph{size of neighboring triangles can vary strongly}, cf.~Fig.~\ref{fig:tri}.
In the next section we prove that, provided each quadrilateral is divided into two triangles properly, the induced surface triangulation is such that the \textit{maximal angle condition} \cite{Babuska76} is satisfied.

\section{The maximal angle condition} \label{geometric}
The family of outer tetrahedral triangulations $\{\mathcal{T}_h\}_{h >0}$  is assumed to be regular, i.e.,  it contains no hanging nodes and the following stability property holds:
\begin{equation}\label{min_angle}
\sup_{h>0}\sup_{S\in\mathcal{T}_h}\rho(S)/r(S)~\le~\alpha<\infty,
\end{equation}
where  $\rho(S)$ and $r(S)$ are the diameters of the smallest ball that contains $S$ and the largest ball contained in $S$, respectively. The stability property implies that the family of tetrahedral triangulations satisfies a \emph{minimum}  (and thus also maximum) \emph{angle condition}: there exists $\theta_{\min} >0$
with
\begin{equation} \label{intangle}
 \frac{\pi}{2}>\theta_{\min}\ge c(\alpha)>0,
\end{equation}
such that all inner angles of all sides of $S \in \mathcal{T}_h$ and all angles between edges of $S$ and their opposite  side are in the interval $[\theta_{\min}, \pi -\theta_{\min}]$.
The constant $c(\alpha)$  depends only on $\alpha$ from \eqref{min_angle}.

Although the surface mesh $\Gamma_h$ induced by $\mathcal{T}_h$ can be  highly shape irregular, the following lemma shows that
 a \textit{maximum angle} property holds.

\begin{lemma}\lbl{theo:2.1}
Assume an outer triangulation $\mathcal{T}_h$ from the regular family $\{\mathcal{T}_h\}_{h >0}$ and let $\Gamma_h$ be consistent with $\mathcal{T}_h$.
There exists $\phi_{\min} >0$, depending only on $\alpha$ from \eqref{min_angle}, such that for every $S\in\mathcal{T}_h$
the following holds:
\begin{itemize}
 \item[a)] if $T=S\cap\Gamma_h$ is a triangular element, then
\begin{equation}\label{max_angle} 0<\phi_{i,T}\leq \pi-\phi_{\min}\ \ i=1,2,3,
\end{equation}
holds, where $\phi_{i,T}$ are the inner angles of the element $T$. 
\item[b)]
if $T=S\cap\Gamma_h$ is a  quadrilateral  element, then
\begin{equation}\label{max_angle_b}
\phi_{i,T}\geq  \phi_{\min},\ \ i=1,2,3,4,
\end{equation}
holds, where $\phi_{i,T}$ are the inner angles of the element $T$.
\end{itemize}										
\end{lemma}
\begin{proof}
 Let $\theta_{\min}$ be the minimal angle bound from \eqref{intangle}. Take $S\in\mathcal{T}_h$.

We first treat the case where $T=S\cap\Gamma_h$ is a triangle $T=BCD$, as illustrated in Fig.~\ref{fig:2.1}.
\begin{figure}[ht!]
\vspace*{0mm}
    \centering
  \resizebox{!}{5.5cm}
    {\includegraphics[ height=60mm]{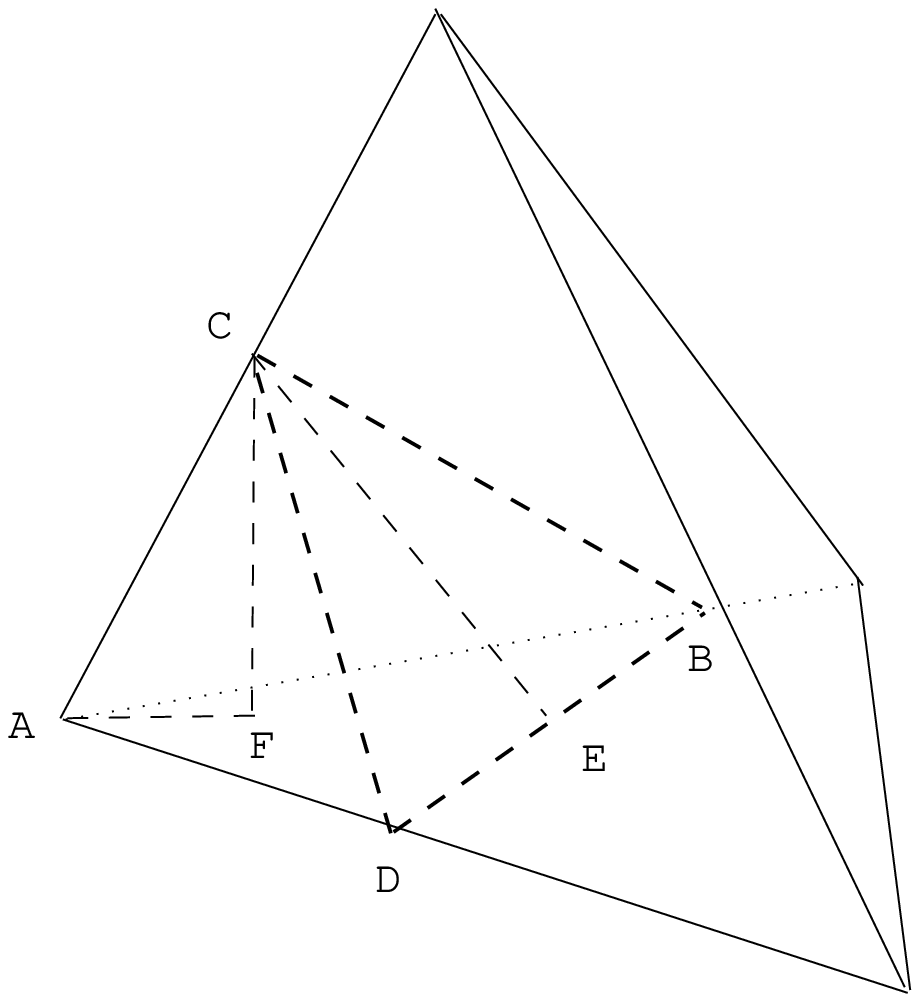}}
    \vspace*{-10mm}
    \caption{ }
   \lbl{fig:2.1}
 \end{figure}
Consider the angle $\phi:=\angle BCD$. Then either $ \phi \leq \pi -\theta_{\min}$ and
\eqref{max_angle} is proved with $\phi_{\min}=\theta_{\min}$ or $\phi \in (\pi -\theta_{\min}, \pi)$. Hence, we treat the latter case.
Note that
\begin{equation*}\frac{|CF|}{|AC|}=\sin(\angle CAF)\geq \sin \theta_{\min} \end{equation*}
and $ \angle BDC < \pi -\phi<\theta_{\min}<\frac{\pi}{2}$.
Take $E$ on the line through $DB$ such that $CE\perp DB$, and $F$ in the plane through $ABD$ such that $CF$ is perpendicular to this plane.
Hence, $|CF|\leq|CE|$ holds.  Using the sine rule we get
\begin{align*}
 \sin(\angle ADC)&=\frac{|AC|}{|CD|}\sin(\angle CAD)\leq \frac{|AC|}{|CD|}\leq \frac{1}{\sin \theta_{\min}} \frac{|CF|}{|CD|}
\leq  \frac{1}{\sin \theta_{\min}} \frac{|CE|}{|CD|} \\ & =  \frac{1}{\sin \theta_{\min}} \sin(\angle BDC) \leq \frac{\sin(\pi-\phi)}{\sin \theta_{\min}}=\frac{\sin(\phi)}{\sin \theta_{\min}} <1.
\end{align*}
Hence, $\angle ADC \leq \arcsin (\frac{\sin \phi}{\sin \theta_{\min}}) \leq 2 \frac{\sin \phi}{\sin \theta_{\min}}$ holds.
This yields
\[ \angle ADB<\angle ADC+\angle CDB\leq  2 \frac{\sin \phi}{\sin \theta_{\min}} + \pi - \phi.\]
With the same arguments we obtain
\[\angle ABD \leq 2 \frac{\sin \phi}{\sin \theta_{\min}} + \pi - \phi. \]
Since $\angle DAB \leq \pi -  \theta_{\min}$  and $\angle DAB  = \pi- (\angle ADB +\angle ABD)$  we  get
\begin{equation} \label{lk}
 \theta_{\min} \leq 4 \frac{\sin \phi}{\sin \theta_{\min}} + 2 \pi - 2 \phi =:g(\phi).
\end{equation}
Since $\phi \in (\pi- \theta_{\min}, \pi) \subset (\frac12 \pi, \pi)$ it suffices to consider $g(\phi)$  for $\phi \in (\frac12 \pi, \pi)$. Elementary computation yields $g(\frac12 \pi) > \theta_{\min}$, $g(\pi)=0$ and $g$ is monotonically decreasing on $(\frac12 \pi, \pi)$. Hence the inequality  \eqref{lk} holds iff $\phi \leq \phi_0$, where $\phi_0 $ is the unique solution in
$(\frac12 \pi, \pi)$ of $g(\phi) = \theta_{\min}$. This proves the result in a).
\medskip

We now consider the case where $T=S\cap\Gamma_h$ is a  quadrilateral $T=ABCD $, as illustrated in Fig.~\ref{fig:2.2}.
\begin{figure}[ht!]
\vspace*{0mm}
    \centering
  \resizebox{!}{5.5cm}
    {\includegraphics[ height=60mm]{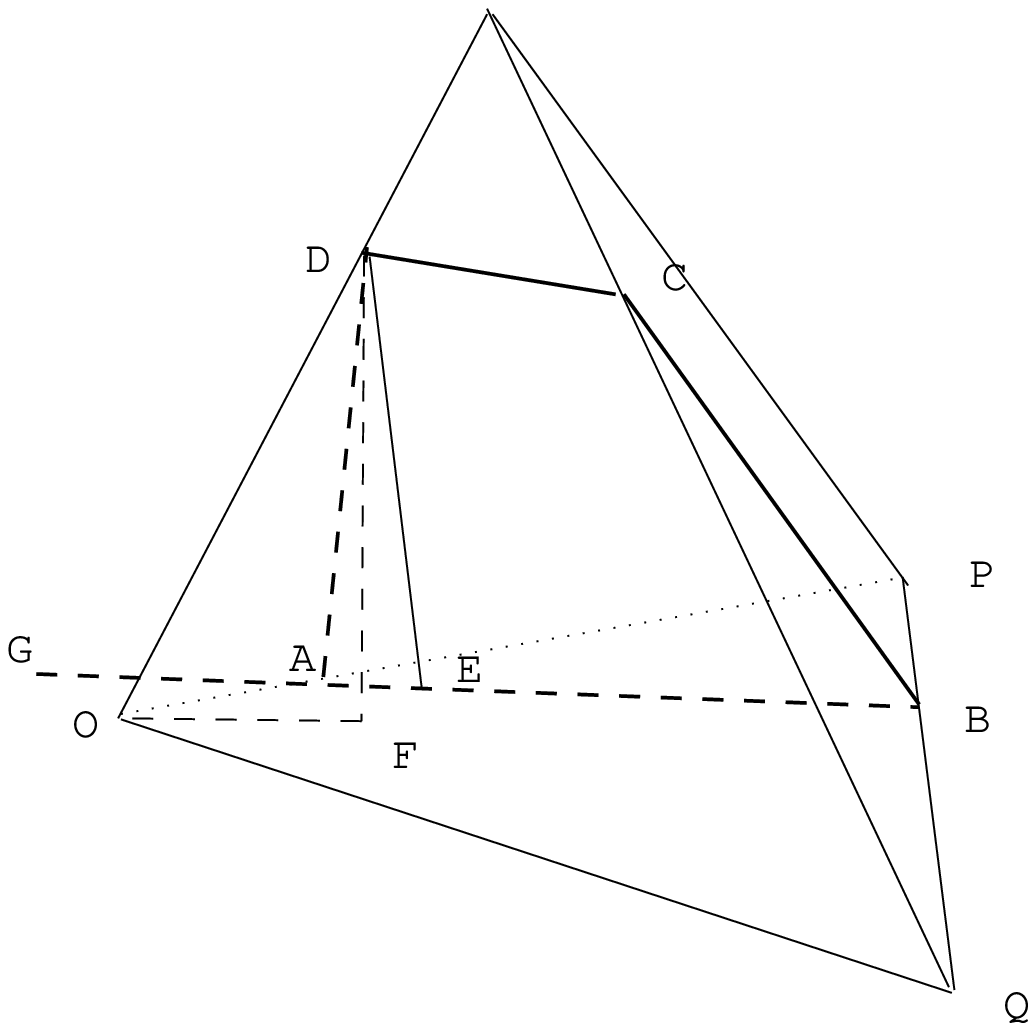}}
    \vspace*{-10mm}
    \caption{ }
   \lbl{fig:2.2}
 \end{figure}
Consider the angle $\phi:=\angle DAB$. Then either $ \phi \in (0 , \theta_{\min})$ or $\phi \in [\theta_{\min}, \pi)$. We only have to  treat the former case.
Take $E$ on the line through $AB$ such that $DE\perp AB$, and $F$ in the plane through $OPQ$ such that $DF$ is perpendicular to this plane.
Hence, $|DF|\leq|DE|$ holds and
\[
 \sin \phi  = \frac{|DE|}{|AD|}. \]
Furthermore, using $
\frac{|DF|}{|OD|}=\sin(\angle DOF)\geq \sin\theta_{\min}$ we get
\begin{align*}
 \sin(\angle OAD)&=\frac{|OD|}{|AD|}\sin(\angle AOD)\leq \frac{|OD|}{|AD|}\leq \frac{1}{\sin\theta_{\min}}  \frac{|DF|}{|AD|}\nonumber\\
&\leq \frac{1}{\sin\theta_{\min}}  \frac{|DE|}{|AD|}= \frac{\sin \phi}{\sin\theta_{\min}} <1.
\end{align*}
This implies
\[ \angle OAD \leq  \arcsin\big( \frac{\sin \phi}{\sin\theta_{\min}} \big) \leq 2 \frac{ \sin \phi}{\sin\theta_{\min}}.
 \]
Hence, since $\angle DAB = \phi \leq 2 \sin \phi$, we obtain
\[ \angle OAB<\angle OAD+\angle DAB \leq \big(1+\frac{1}{\sin\theta_{\min}}\big) 2 \sin \phi .\]
 Using $\angle OAB = \pi - \angle PAB $ and $\angle PAB < \pi - \angle OPQ < \pi - \theta_{\min}$
results in
\begin{equation} \label{lkk}
  \theta_{\min} < \big(1+\frac{1}{\sin\theta_{\min}}\big) 2  \sin \phi.
\end{equation}
For $\phi \in (0,\theta_{\min})$ the inequality  \eqref{lkk} holds iff $\phi \geq \phi_0$, where $\phi_0 $ is the unique solution in $(0,\frac12 \pi)$ of
$\theta_{\min} = \big(1+\frac{1}{\sin\theta_{\min}}\big) 2  \sin \phi_0$. Thus the result in b) holds.\\
\end{proof}

The lemma readily yields the following result.

\begin{theorem}[maximum angle condition] \label{cor1}
Consider a regular family of tetrahedral triangulations $\{ \T_h\}_{h>0}$ and a surface triangulation $\Gamma_h= \cup_{T \in \mathcal{F}_h} T$ that is consistent with $ \T_h$.
Assume that any quadrilateral element $T=S\cap\Gamma_h$, $S\in\mathcal{T}_h$, is divided in two triangles
by  connecting the vertex with largest inner
angle with its opposite vertex.  The resulting surface triangulation satisfies the following \emph{maximal angle condition}.  There exists $\phi_{\min}>0$ depending only on $\alpha$ from
 \eqref{min_angle} such that:
\begin{equation}
\label{max_angle2} 0<\sup_{T\in\mathcal{F}_h}\phi_{i,T}\leq \pi-\phi_{\min}\ \ i=1,2,3,
\end{equation}
 where $\phi_{i,T}$ are the inner angles of the element $T$.
\end{theorem}
\begin{proof} If $T=S\cap\Gamma_h$ is a triangle, then \eqref{max_angle2} directly follows from \eqref{max_angle}.
 Let $T=S\cap\Gamma_h$ be a quadrilateral, with its four inner angles denoted  by $\theta_4 \geq \theta_3 \geq \theta_2 \geq \theta_1 > 0$.
 From the result in \eqref{max_angle_b} we have $\theta_i \geq \phi_{\rm min}$ for all $i$.
The vertex with angle $\theta_4$ is connected with the opposite vertex. Let $T_1$ be one of the resulting triangles.
One of the angles of $T_1$ is $\theta_j$ with $j\in\{1,2,3\}$. From $ \theta_j \geq  \phi_{\rm min}$
 it follows that the other two angles are both bounded by $\pi - \phi_{\rm min}$.
Furthermore, from $\theta_j = 2\pi - \theta_4 - \sum_{i=1, i \neq j}^3 \theta_i \leq 2 \pi - \theta_j - 2 \phi_{\rm min}$
 it follows that $\theta_j \leq\pi - \phi_{\rm min}$ holds.
\end{proof}
\ \\[1ex]
In the remainder we assume that quadrilaterals are subdivided in the way as explained in Theorem~\ref{cor1}. Hence, the inner angles in the surface triangulation $\mathcal{F}_h$ are bounded by a constant $\theta^\ast < \pi$ that depends only on the stability (close to $\Gamma$) of the outer tetrahedral triangulation $\mathcal{T}_h$. In particular $\theta^\ast$ is \emph{independent of $h$ and of how  $\Gamma_h$ intersects the outer triangulation $\mathcal{T}_h$}.

\section{Application in a finite element method}\label{app:fem}
In this section, we use  the maximum angle property of the surface triangulation to derive an optimal finite element interpolation result.
On $\mathcal{F}_h$ we consider the space of linear finite element functions:
\begin{equation} \label{defFE}
 V_h=\{v_h \in \mathcal C(\Gamma_h) : v_h \in \mathcal P_1(T)\quad \text{for all}~~ T\in\mathcal F_h \}.
\end{equation}
This finite element space is the same as the one studied by Dziuk in \cite{Dziuk88}, but an important  difference
is that in the approach in  \cite{Dziuk88} the triangulations have to be shape regular. In general, the finite element space $V_h$ is different from the
surface finite element space constructed in \cite{OlshanskiiReusken08,Reusken08}.

Below we derive an approximation result for the finite element space $V_h$. Since the discrete surface  $\Gamma_h$ varies with $h$, we have to explain in which sense $\Gamma_h$ is close to $\Gamma$. For this we use a standard setting applied in the analysis of discretization methods for partial differential equations on surfaces, e.g.~\cite{Demlow06,Dziuk88,Dziuk07,Dziuk011,Reusken08}.

Let $U:=\{\, x \in \R^3~|~{\rm dist}(x,\Gamma) < c\, \}$ be a sufficiently small neighborhood of $\Gamma$. We define $\T_h^\Gamma:=\{ \, T \in  \Th~|~ {\rm meas}_2(T \cap \Gamma_h) > 0\, \}$, i.e., the collection of tetrahedra which intersect
the discrete surface ${\Gamma}_h$, and assume that  $\T_h^\Gamma \subset U$.
Let $d$ be the signed distance function to $\Gamma$, with $d <0$ in the interior of $\Gamma$,
\[
  d: U \rightarrow \R,\qquad  |d(x)|:={\rm dist}(x,\Gamma) \quad\mbox{for all $x \in U$}.
\]
Thus $\Gamma$ is the zero level set of $d$. Note that $\bn_{\Gamma} =\nabla d$ on $\Gamma$. We define $\bn(x):=\nabla d(x)$ for  $x \in U$. Thus $\bn$ is the outward pointing normal on $\Gamma$ and $\|\bn (x)\|=1$ for all $x\in U$. Here and in the remainder    $\|\cdot\|$ denotes the Euclidean norm on $\R^3$.
We introduce a local orthogonal coordinate system by using the  projection $\bp:\, U \rightarrow\Gamma$:
\[
 \bp(x)=x-d(x)\bn(x) \quad \text{for all } x \in U.
\]
We assume that the decomposition $x=\bp(x)+ d(x)\bn(x)$ is unique for all $x \in U$. Note that
$
  \bn(x)= \bn\big(\bp(x)\big)$ for all   $x \in U$.
For a function $v$ on $\Gamma$, its extension is defined as
\begin{equation}
 v^e(x):= v(\bfp(x)), \quad \hbox{for all } x\in U.
\end{equation}
The outward pointing (piecewise constant) unit normal on $\Gamma_h$ is denoted by $\bn_h$.
Using this local coordinate system we introduce the following assumptions on $\Gamma_h$:
\begin{align}
  & \bfp: \Gamma_h \to \Gamma \quad \text{is bijective}, \label{cond1} \\
  & \max_{x \in\Gamma_h} |d(x)| \lesssim  h^2, \label{cond2} \\
  &  \max_{x \in\Gamma_h} \| \bn(x)- \bn_h(x)\| \lesssim h, \label{cond3}
\end{align}
 where $h=\sup_{T\in\T_h^\Gamma}\rho(T)$.
In \eqref{cond2}-\eqref{cond3} we use the common notation, that the inequality holds with a constant independent of $h$. In \eqref{cond3}, only $x\in\Gamma_h$  are considered for which $\bn_h(x)$ is well-defined.
Using these assumptions, the following result is derived in \cite{Dziuk88}.
\begin{lemma}\lbl{lem:norms}
 For any function $u\in H^2(\Gamma)$, we have, for arbitrary $T \in \mathcal{F}_h$ and $\tilde T:=\bfp (T)$:
\begin{align}
\|u^e\|_{0,T} &\simeq  \|u\|_{0,\tilde T}, \\
|u^e|_{1,T} &\simeq  |u|_{1,\tilde T},\\
|u^e|_{2,T} &\lesssim  |u|_{2,\tilde T}+ h|u|_{1,\tilde T}, \label{ress3}
\end{align}
 where $A \simeq B$ means $B\lesssim A\lesssim B$ and the constants in the inequalities are independent of $T$ and of $h$.
\end{lemma}

\subsection{Finite element interpolation error}
Based on the results in Lemma~\ref{lem:norms}, the maximum angle property and the approximation results derived in \cite{Babuska76} we easily obtain an optimal bound for the interpolation  error in the space $V_h$.
Consider the standard finite element nodal interpolation  $I_h: C(\Gamma_h) \to V_h$:
\begin{equation}
 (I_h v)(x) = v(x),\quad \text{for all}~~ x \in \mathcal{V},
\end{equation}
with $\mathcal{V}$ the set of vertices of the triangles in $\Gamma_h$.
\begin{theorem} \label{thm2}
For any $u\in H^2(\Gamma)$  we have
\begin{align}
\|u^e- I_h u^e\|_{L^2(\Gamma_h)} &\lesssim h^2 \|u\|_{H^2(\Gamma)}, \label{r1} \\
\|u^e- I_h u^e\|_{H^1(\Gamma_h)} & \lesssim h \|u\|_{H^2(\Gamma)}. \label{r2}
\end{align}
\end{theorem}
\begin{proof}
 From standard interpolation theory we have
\[
  \|u^e- I_h u^e\|_{L^2(T)} \lesssim h^2 |u^e|_{2,T},
\]
where the constant in the upper bound is independent of (the shape of) $T$.
Using the result in \eqref{ress3} and summing over $T \in \mathcal{F}$ proves the result \eqref{r1}.
For the interpolation error bound in the $H^1$-norm we use the results from \cite{Babuska76}. For the interpolation error bounds derived in that paper the maximum angle property is essential.
 From \cite{Babuska76} we get
\begin{equation*}
 \|u^e- I_h u^e\|_{H^1(T)}\lesssim h \|u\|_{H^2(T)}.
\end{equation*}
Due to the maximum angle property the constant in the upper bound is independent of $T$.
Using the results in Lemma~\ref{lem:norms} and summing over $T\in\mathcal{F}_h$ we obtain the result \eqref{r2}.
\end{proof}
\ \\[1ex]
If one considers an $H^1(\Gamma)$ elliptic partial differential equation on $\Gamma$, the error for its finite element discretization in the surface space $V_h$ can be analyzed along the same lines as in \cite{Dziuk88}.  A difference with  the planar case is that geometric errors arise due to the approximation of $\Gamma$ by $\Gamma_h$. Using the interpolation error bounds in Theorem~\ref{thm2} and bounding the geometric errors, with the help of the assumptions \eqref{cond1}-\eqref{cond3}, results in optimal order discretization error bounds.

\subsection{Conditioning of the mass matrix}\label{s_cond}
Clearly the (strong) shape irregularity of the surface triangulation will influence the conditioning of the mass and stiffness matrices. Let $N$ be the number of vertices in the surface triangulation and $\{\phi_i\}_{i=1}^{N}$  the nodal basis  of the finite element space $V_h$. The mass and stiffness matrices are given by
\begin{align}
 \mathbf{M} & =(m_{ij})_{i,j=1}^N,\quad \hbox{ with }\quad m_{ij}=\int_{\Gamma_h}\phi_i\phi_j\, ds, \\
\mathbf{A} &=(a_{ij})_{i,j=1}^N,\quad \hbox{ with }\quad a_{ij}=\int_{\Gamma_h}\nabla_{\Gamma_h}\phi_i\nabla_{\Gamma_h}\phi_j\, ds.
\end{align}
We also need their scaled versions.
Let $\mathbf{D}_M$ and $\mathbf{D}_A$ be the diagonals of $\mathbf{M}$ and $\mathbf{A}$, respectively. The scaled matrices are denoted by
\begin{equation} \label{scaledmat}
\mathbf{M}^s=\mathbf{D}_M^{-\frac{1}{2}}\mathbf{M}\mathbf{D}_M^{-\frac{1}{2}}, \quad \mathbf{A}^s=\mathbf{D}_A^{-\frac{1}{2}}\mathbf{A}\mathbf{D}_A^{-\frac{1}{2}}.
\end{equation}
 From a simple scaling argument it follows that the spectral condition number of $\mathbf{M}^s$ is bounded uniformly in $h$ and in the shape (ir)regularity of the surface triangulation. For completeness we include a proof.
\begin{theorem} \label{thmmass}
 The following holds:
\[
 \frac{2}{\sqrt{2}+2} \leq \frac{\langle \mathbf{M} \bv, \bv\rangle}{\langle \mathbf{D}_M \bv, \bv\rangle} \leq 4 \quad \text{for all}~~ \bv \in \Bbb{R}^N, ~ \bv \neq 0.
\]
\end{theorem}
\begin{proof} The set of all vertices in $\mathcal{F}_h$ is denoted by $\mathcal{V}=\{\, \xi_i~|~1 \leq i \leq N\,\}$. Let $\bv \in\Bbb{R}^N$ and $v_h \in V_h$ be related by $v_h= \sum_{i=1}^N v_i \phi_i$, i.e., $v_i=v_h(\xi_i)$. Consider a triangle $T \in \mathcal{F}_h$ and let its three vertices be denoted by $\xi_1,\xi_2,\xi_3$. Using quadrature we obtain
\[ \begin{split}
   \int_{T} v_h(s)^2 \, ds  &= \frac{|T|}{3} \big( \frac14 ( v_1+v_2)^2 +  \frac14 ( v_2+v_3)^2 +  \frac14 ( v_3+v_1)^2 \big) \\
 & = \frac{|T|}{6} \big( v_1^2 +v_2^2 +v_3^2 +v_1v_2+v_2v_3+v_3v_1\big).
   \end{split}
\]
Hence, $\int_{T} v_h(s)^2 \, ds \leq \frac{|T|}{3} \sum_{i=1}^3 v_i^2$ holds.
 From a sign argument it follows that at least one of the three terms $v_1v_2$, $v_2v_3$ or $v_3v_1$ must be positive.
 Without loss of generality we can assume $v_1v_2 \geq 0$. Using $|v_2v_3+v_3v_1| \leq \frac{1}{\sqrt{2}} \big( v_1^2 +v_2^2 +v_3^2\big)$ we get
\[
 \int_{T} v_h(s)^2\, ds \geq \frac{|T|}{6} \big( v_1^2 +v_2^2 +v_3^2 - \frac{1}{\sqrt{2}} ( v_1^2 +v_2^2 +v_3^2) \big) = \frac{|T|}{6(\sqrt{2}+2)}\big( v_1^2 +v_2^2 +v_3^2 \big).
\]
Note that $\langle \mathbf{M} \bv, \bv\rangle = \int_{\Gamma_h} v_h(s)^2 \, ds = \sum_{T \in \mathcal{F}_h} \int_{T} v_h(s)^2 \, ds$,
and thus we obtain, with $\mathcal{V}(T)$ the set of the three vertices of $T$,
\begin{equation}\label{p1}
 \frac{2}{\sqrt{2}+2} \frac{1}{12} \sum_{T\in \mathcal{F}_h}|T| \sum_{\xi \in \mathcal{V}(T)} v_h(\xi)^2 \leq \langle \mathbf{M} \bv, \bv\rangle \leq 4 \frac{1}{12} \sum_{T\in \mathcal{F}_h}|T| \sum_{\xi \in \mathcal{V}(T)} v_h(\xi)^2.
\end{equation}
We observe that
\begin{equation} \label{p2}
 \frac{1}{12} \sum_{T\in \mathcal{F}_h}|T| \sum_{\xi \in \mathcal{V}(T)} v_h(\xi)^2 = \frac{1}{12} \sum_{i=1}^N |{\rm supp}(\phi_i)| v_i^2
\end{equation}
holds. From the definition of $ \mathbf{D}_M$ it  follows that
\begin{equation} \label{p3}
 \begin{split}
  \langle \mathbf{D}_M \bv, \bv\rangle & = \sum_{i=1}^N \int_{\Gamma_h} \phi_i^2 \, ds \, v_i^2 = \sum_{i=1}^N v_i^2 \sum_{T \in {\rm supp}(\phi_i)} \int_T \phi_i^2 \, ds  \\
 & = \sum_{i=1}^N  v_i^2 \sum_{T \in {\rm supp}(\phi_i)} \frac{|T|}{12}   = \frac{1}{12} \sum_{i=1}^N  |{\rm supp}(\phi_i)| v_i^2.
 \end{split}
\end{equation}
Combination of the results in \eqref{p1}, \eqref{p2} and \eqref{p3} completes the proof.
\end{proof}

\subsection{Conditioning of the stiffness matrix}\label{s_condS}
We finally address the issue of conditioning of the diagonally scaled stiffness matrix  $\mathbf{A}^s$, cf.~\eqref{scaledmat}.
This matrix has a one dimensional kernel due to the constant nodal mode.
Thus, we  consider the effective condition number $\cond(\mathbf{A}^s)=\lambda_{\max}(\mathbf{A}^s)/\lambda_2(\mathbf{A}^s)$, where $\lambda_2$ is the minimal
nonzero eigenvalue.
We shall argue below that the condition number of $\mathbf{A}^s$ can not be bounded in general by a constant
dependent exclusively  on $\mathcal{T}_h$, but not on $\Gamma_h$.
Indeed, assume a smooth closed surface $\Gamma$, with $|\Gamma|=1$, and a smooth function $u$ defined on $\Gamma$, such that
$\|\nabla_\Gamma u\|_{L^2(\Gamma)}=\|u\|_{H^2(\Gamma)}=1$. Let $\Gamma_h$ be the zero level of the piecewise linear Lagrange interpolant of the
signed distance function to $\Gamma$.  Denote $u_h=I_h u^e$, as in  Theorem~\ref{thm2}, and $\bv=(v_1,\dots,v_N)^T$ is the corresponding vector of nodal values. From the result in \eqref{r2} we obtain
\begin{equation}\label{aux1}
\langle \mathbf{A} \bv, \bv\rangle=\|\nabla_{\Gamma_h} u_h\|_{L^2(\Gamma_h)}=1+O(h).
\end{equation}
On the other hand, if there is a node $\xi$ in the volume triangulation  $\mathcal{T}_h$ such that $\mbox{dist}(\xi,\Gamma_h)<\eps\ll1$, then there can appear a triangle in $\mathcal{F}_h$ with a minimal
angle of $O(\eps)$. This implies that there is a diagonal element in $\mathbf{A}$ of order $O(\eps^{-1})$.
Without lost of generality we may assume $A_{11}=O(\eps^{-1})$ and $v_1=1$. Thus we get
\begin{equation}\label{aux2}
\langle \mathbf{D}_A \bv, \bv\rangle\ge A_{11} v_1^2=O(\eps^{-1}).
\end{equation}
Comparing \eqref{aux1} and \eqref{aux2} we conclude that  $\cond(\mathbf{A}^s)\ge O(\eps^{-1})$, with $\eps\to0$. Results of numerical experiments in the next section demonstrate that the blow up of $\cond(\mathbf{A}^s)$
can be seen in some cases.

One might also be interested in a more general dependence of the eigenvalues of
$\mathbf{A}^s$  on the distribution of tetrahedral nodes in $\mathcal{T}_h$ in a neighborhood of $\Gamma_h$.
To a certain extend this question is addressed in \cite{OlshanskiiReusken08}.


\section{Numerical experiment} \label{secExp}
In this section we present a few results of numerical experiments which illustrate the interpolation estimates from Theorem~\ref{thm2} and the conditioning of mass and stiffness matrices.
Assume the surface $\Gamma$, which is the unit sphere $\Gamma=\{ \, x\in \mathbb{R}^3~|~ \|x\|=1\, \}$, is
embedded in the bulk domain  $\Omega=[-2,2]^3$.
The  signed distance function to $\Gamma$ is denoted by $d$.
We construct a hierarchy of uniform tetrahedral triangulations $\{\mathcal{T}_h\}$ for $\Omega$, with $h\in\{1/2,1/4,1/8,1/16,1/32\}$.
Let $d_h$ be the piecewise nodal Lagrangian interpolant  of $d$. The triangulated surface  is given by
\[
\Gamma_h= \bigcup_{T \in \mathcal{F}_h}T = \{ \, x\in \Omega~|~d_h(x)=0\, \}.
\]
The corresponding finite element space $V_h$ consists of all piecewise affine functions with respect to 
$\mathcal{F}_h$, as defined in \eqref{defFE}.
For $h\in\{1/2,1/4,1/8,1/16,1/32\}$, the resulting dimensions of $V_h$ are
$N=164,812,3500,14264,57632$, respectively. In agreement with the 2D nature of $\Gamma_h$, we have $N\sim h^{-2}$.

To illustrate the result of Theorem~\ref{thm2}, we present   the interpolation errors $\|u^e-I_h u^e\|_{L^2(\Gamma_h)}$ and $|u^e-I_h u^e|_{1,\Gamma_h}$ for the smooth function
$$
u(x)=\frac{1}{\pi}x_1 x_2\arctan({2x_3})
$$
defined on the unit sphere, with $x=(x_1,x_2,x_3)^T$. The dependence of the interpolation errors
on the number of  degrees of freedom $N$ is shown in Figure~\ref{fig:err} (left). We observe the optimal
error reduction behavior, consistent with the estimates in \eqref{r1}, \eqref{r2}.
\begin{figure}
    \centering
    \includegraphics[width=0.46\textwidth]{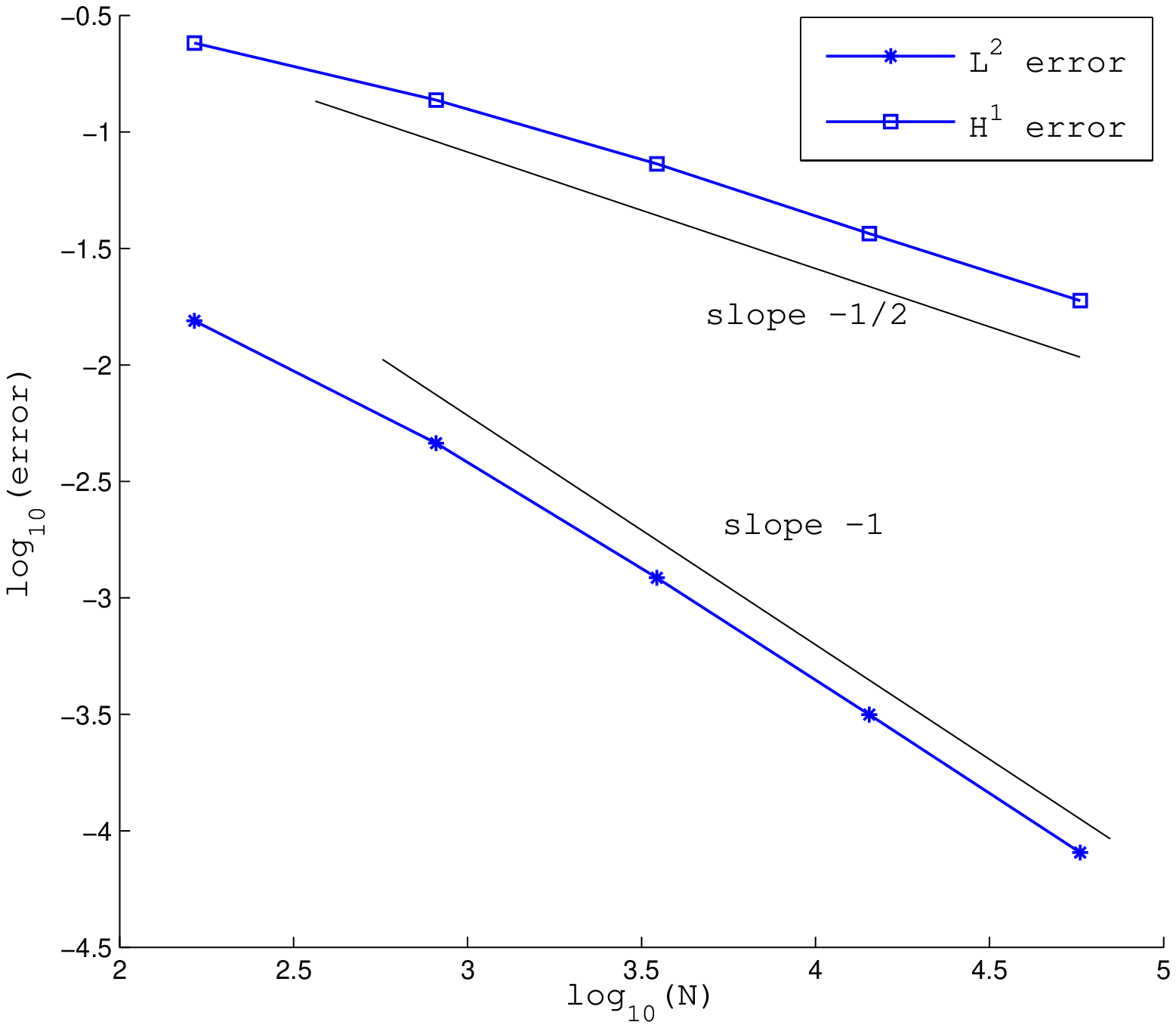}\hfill \includegraphics[width=0.5\textwidth]{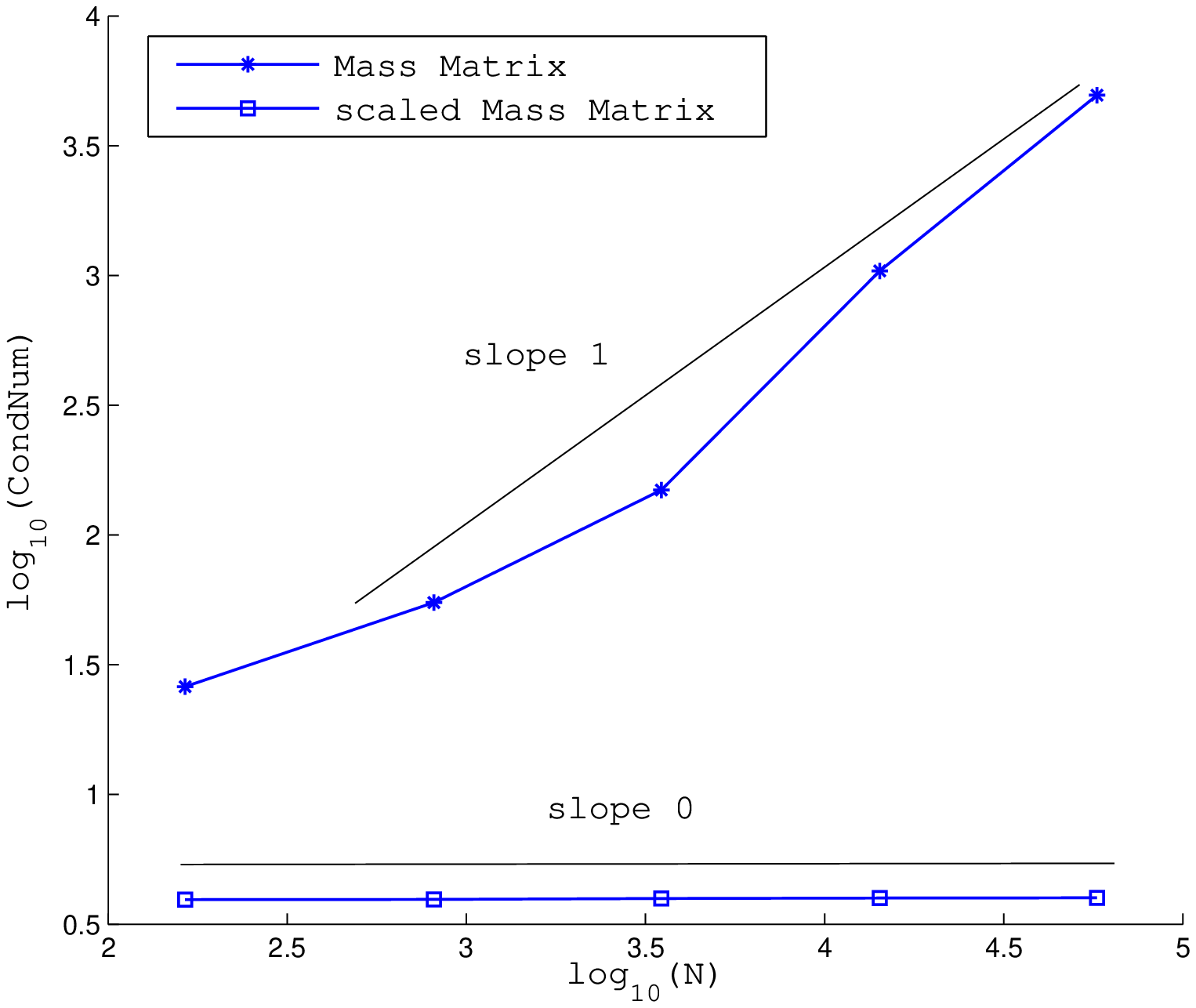}
    \caption{Left: Interpolation error as a function of \# d.o.f.; Right: The condition number of the mass matrix as a function of of \# d.o.f.}
   \lbl{fig:err}
 \end{figure}

Further, for the same sequence of meshes we compute the spectral  condition numbers  of the mass  matrix $\mathbf{M}$ and the diagonally scaled mass matrix $\mathbf{M}^s$.
The dependence of  the condition numbers on the number of  degrees of freedom $N$ is illustrated in Figure~\ref{fig:err} (right).
As was proved in  Theorem~\ref{thmmass}, the scaled mass matrix has a uniformly bounded condition number.

We discussed in section~\ref{s_condS} that concerning the effective condition number of the scaled stiffness matrix the situation  is
more delicate. To illustrate this,  we performed  an experiment in which the intersection between a fixed outer triangulation  and the surface is varied.
Let $\Gamma$ be the boundary of the unit sphere   with the center located in $(0,0,z_c)$.
The discrete surface $\Gamma_h$ is defined as described above,  induced by the uniform outer triangulation. We choose a fixed outer triangulation with $h=1/16$.  We now consider different values for $z_c$, thus ``moving the surface through the outer triangulation''.
The $z_c$ values are given in the first column of Table~\ref{tab:00}. Note that for the largest shift $z_c=0.03$ we have $z_c \approx  0.5 h$. 

For the different surface triangulations we computed the interpolation errors  as described above. It turns out that for the values $z_c \neq 0$
 the error behavior is essentially the same as that for $z_c=0$ (illustrated in Fig.~\ref{fig:err}). 
 
 In the second to fourth columns of Table~\ref{tab:00} two geometry related quantities are given. The second column shows the value of
 the maximum angle occuring in the surface triangulation. Consistent to the theory, cf. Theorem~\ref{cor1}, the maximum angle is bounded
away from $180^\circ$. Small angles, however, can occur. In the third and fourth column we show the value of the minimum angle
and  the number of triangles in the surface triangulation with the smallest angle smaller than $1^\circ$. 
As expected, both the minimal angle and this number of small angles strongly varies depending on $z_c$. For $z_c=0$ the smallest angle
 in the surface triangulation has value $\phi_{\rm min}=1.85^\circ$. Extremely small angles can occur, e.g. for $z_c=0.00005$, we have
$\phi_{\rm min}=8.54$e-$7^\circ$.
 The dimension and the effective condition number of the scaled stiffness matrix  $\mathbf{A}^s$
are given in the fifth and sixth column of Table~\ref{tab:00}. The values of the condition number  show a strong dependence
on the sphere location (value of $z_c$). These large condition numbers indicate that linear systems with these matrices may be
 hard to solve using an iterative method. To investigate this further, we used the standard PCG MATLAB solver with ILU(0) preconditioner.
For given $\bv$, we computed $\mathbf{b}=\mathbf{A}^s \bv$ and  applied  the MATLAB PCG iterative solver with a relative residual tolerance of $10^{-8}$.
The resulting iteration numbers are given in column 7 of Table~\ref{tab:00}. These iteration counts are ``high'' compared to the ones that are generally
 needed for standard discretization of diffusion problems.
To make this more quantitative, we constructed a reference matrix $\bA^{\rm ref}$ as follows: $ \bA^{\rm ref} ={\rm blocktridiag}(-\bB^T,\bD,-\bB)$,
 with $\bD={\rm tridiag}(-1,6,-1)$, $\bB={\rm tridiag}(0,1,1)$. In most rows the matrix $\bA^{\rm ref}$ has 7 nonzero entries, which is approximately
 the same as the average number of nonzero entries per row in the matrix $\mathbf{A}^s$ used in the experiment.
 In $ \bA^{\rm ref}$ we use 120 blockrows and blockcolums and the matrices $\bD$ and $\bB$ have dimension 120.
Then the matrix $\bA^{\rm ref}$ has dimension 14400, which is comparable to the dimension of $\mathbf{A}^s$ used
 in the experiment, cf. Table~\ref{tab:00}. The same iterative solver with the same stopping criterion applied to a
 linear system with  $ \bA^{\rm ref}$ resulted in 42 PCG iterations, which is much lower than the
iteration numbers listed in Table~\ref{tab:00}. \\
In view of these observations, and the fact that solving a PDE on a surface (in 3D) is a \emph{two}-dimensional
problem, it is better to use a direct solver. We performed experiments with the MATLAB sparse direct solver $\bA^s \setminus \mathbf{b}$.
 We measured computing time by the MATLAB function {\sc cputime}. For the system with the reference matrix $ \bA^{\rm ref}$ we obtained
(on our machine) {\sc cputime}$=1.38$. For the matrix $\mathbf{A}^s$ we obtained CPU time measurements given in the last column of Table~\ref{tab:00}.
 These show that for the direct MATLAB solver the matrices  $\mathbf{A}^s$ are not (much) more difficult to deal with than the reference matrix $  \bA^{\rm ref}$.  Variations in CPU times are probably caused by slightly different fill-in properties of matrices for different grids.
The one dimensional kernel of the matrix $\mathbf{A}^s$ did not cause difficulties for the solver.
We checked the  accuracy of the computed solution (in the energy norm) and this was satisfactory.
\begin{table}[ht!]\small
\centering
\begin{tabular}[l]{|l|c|c|c||c|c|c|c|}
 $z_c$ &{$\phi_{\rm max}$}&{$\phi_{\rm min}$} &{$ \# T : \atop \phi_{\rm min} <1^\circ  $} & {\rm dim}$(\mathbf{A}^s)$ & $\cond(\mathbf{A}^s)$ & $\#~$PCG & {\sc cputime} \\[0.8ex]
\hline
 0.03  & $147.4^\circ$ & $0.050^\circ$ & 420  &14406  &1.82e+4 &245   & 3.64  \\
 0.02  & $145.3^\circ$ & $0.027^\circ$ &292  &14376  &2.20e+4 &282   & 3.52  \\
 0.008  & $145.4^\circ$ &$0.014^\circ$ & 270  &14368  &3.44e+4 &331   & 3.61  \\
0.002   & $144.3^\circ$ &$0.002^\circ$ &126  &14300 &1.94e+5 &285  & 2.33\\
 0.0005 & $141.0^\circ$ &1.22e-4$^\circ$ & 20 &14288 &3.07e+6 &259  &  1.93\\
 0.00025& $140.4^\circ$ &3.05e-5$^\circ$ & 20  &14288 &1.23e+7 &191  &  2.22 \\
 0.00005& $139.9^\circ$ &8.54e-7$^\circ$ & 24  &14288 &3.06e+8 &202  & 1.43 \\
    0   & $139.8^\circ$ &$1.85^\circ$ & 0   &14264 &9.14e+3 &142  & 2.85 \\
\hline
\end{tabular}
\caption{Angles in the surface triangulation, dimension of $\mathbf{A}^s$,  $\cond(\mathbf{A}^s)$, iteration count for PCG and timing for direct solver. \label{tab:00}}
\end{table}

\section{Conclusions} The main new result of this paper is a geometric property of the piecewise planar surface which is the zero level of a continuous piecewise affine level set function. If this piecewise planar surface is consistent with an outer tetrahedral triangulation that satisfies the \emph{minimum angle condition}, then after a suitable subdivision of the quadrilaterals into two triangles the resulting surface triangulation satisfies a \emph{maximum angle condition}. This maximum angle property of the surface triangulation is used to derive optimal error bounds for the nodal interpolation operator in the finite element space of continuous piecewise linear functions on the surface triangulation. This implies that the discretization of a surface diffusion PDE in this finite element space results in optimal discretization error bounds. We study the conditioning of the scaled mass and stiffness matrices corresponding to this finite element space. The condition number of the scaled mass matrix is shown to be uniformly bounded. The scaled stiffness matrix can have a very large effective condition number. Results of a numerical experiment indicate that for solving systems with the scaled stiffness matrix it is better to use a sparse direct solver rather than an iterative solver. A topic that we plan to investigate further is whether some grid smoothing (elimination of extremely small angles) can be developed such that the optimal approximation property still holds and the conditioning of the  scaled stiffness matrix is improved.

\section*{Acknowledgments}
The authors thank the referees for their comments,  which have led to significant improvements of the original version of this paper.
This work has been supported in part by the DFG  through grant RE1461/4-1 and the Russian Foundation for  Basic Research through grants 12-01-91330, 12-01-00283.

\bibliographystyle{abbrv}
\bibliography{literatur}
\end{document}